\documentclass[11pt]{article}

\usepackage{amsmath, amssymb}
\usepackage{authblk}
\usepackage{geometry}
\usepackage{hyperref}
\usepackage{graphicx}
\usepackage{epstopdf}
\usepackage{mathrsfs}
\usepackage{amsfonts}
\usepackage{amscd,amsthm,bm,psfrag}
\usepackage[numbers,sort&compress]{natbib}

\makeatletter
\renewcommand{\@seccntformat}[1]{{\csname the#1\endcsname}{\normalsize.}\hspace{.5em}}

\makeatother

\numberwithin{equation}{section}

\newtheorem{thm}{Theorem}[section]

\newtheorem{cla}{Claim}

\newtheorem*{thm*}{Theorem}
\newtheorem*{lem*}{Lemma}
\newtheorem*{prop*}{Proposition}

\newcounter{cases}
\newcounter{subcases}[cases]

\newcommand\ex{\ensuremath{\mathrm{ex}}}

\newcommand\cB{{\mathcal B}}
\newcommand\cC{{\mathcal C}}

\newcommand\cE{{\mathcal E}}
\newcommand\cF{{\mathcal F}}

\newcommand\cH{{\mathcal H}}

\newcommand\cM{{\mathcal M}}

\begin{document}
\title{The Tur\'{a}n number of Berge matchings}
% \author{
% Xiamiao Zhao\thanks{\small Al. Email: \small \texttt{g}.}\,, \hspace{0.1em}
% Zixuan Yang\thanks{\small Al. Email: \small \texttt{g}.}\,, \hspace{0.1em}
% Yichen Wang\thanks{\small Al. Email: \small \texttt{g}.}\,, \hspace{0.1em}
% Yuhang Bai\thanks{\small Al. Email: \small \texttt{g}.}\,, \hspace{0.1em}
% Junpeng Zhou\thanks{\small \textit{Corresponding author.} Department of Mathematics, Shanghai University, Shanghai 200444, P.R. China. Email: \small \texttt{junpengzhou@shu.edu.cn}.} \thanks{\small Newtouch Center for Mathematics of Shanghai University, Shanghai 200444, P.R. China.}
% }

\author[1]{\small\bf Yichen Wang\thanks{E-mail: wangyich22@mails.tsinghua.edu.cn}}
\author[2,3]{\small\bf Zixuan Yang\thanks{E-mail: yangzixuan@nwpu.edu.cn}}
\author[1]{\small\bf Xiamiao Zhao\thanks{E-mail: zxm23@mails.tsinghua.edu.cn}}
\author[2,3]{\small\bf Yuhang Bai\thanks{E-mail: yhbai@mail.nwpu.edu.cn}}
\author[4,5]{\small\bf Junpeng Zhou\thanks{\textit{Corresponding author}. E-mail: junpengzhou@shu.edu.cn}}

\affil[1]{\small Department of Mathematical Sciences, Tsinghua University, Beijing 100084, P.R. China.}
\affil[2]{\small School of Mathematics and Statistics, Northwestern Polytechnical University, Xi'an 710129, Shaanxi, P.R. China.}
\affil[3]{\small Xi'an-Budapest Joint Research Center for Combinatorics, Xi'an 710129, Shaanxi, P.R. China.}
\affil[4]{\small Department of Mathematics, Shanghai University, Shanghai 200444, P.R. China.}
\affil[5]{\small Newtouch Center for Mathematics of Shanghai University, Shanghai 200444, P.R. China.}

\date{}

\maketitle

\begin{abstract}
    Given a graph $F$, an $r$-uniform hypergraph $\mathcal{H}$ is a {\em Berge-$F$} if there is a bijection $\phi:E(F)\to E(\mathcal{H})$ such that $e\subseteq \phi(e)$ for each $e\in E(F)$. Given a family $\mathcal{F}$ of $r$-uniform hypergraphs, an $r$-uniform hypergraph is $\mathcal{F}$-free if it does not contain any member of $\mathcal{F}$ as a subhypergraph. The Tur\'{a}n number of $\mathcal{F}$ is the maximum number of hyperedges in an $\mathcal{F}$-free $r$-graph on $n$ vertices. Let $M_{s+1}$ denote a matching of size $s+1$, i.e., the graph consisting of $s+1$ independent edges. Khormali and Palmer [\textit{European J. Combin.} 102 (2022) 103506] completely determined the Tur\'{a}n number of Berge matchings for sufficiently large $n$. Subsequently, Kang, Ni, and Shan [\textit{Discrete Math.} 345 (2022) 112901] determined the exact value of the Tur\'{a}n number of Berge-$M_{s+1}$ for all $n$ when $r \le s-1$ or $r \ge 2s+2$. In this paper, we settle the final open case $s \le r \le 2s+1$, thereby completing the determination of the Tur\'{a}n number of Berge matchings.
    %  Moreover, we establish several exact and general results on the Tur\'{a}n numbers of Berge matchings together with a single $r$-graph, as well as of Berge matchings together with Berge bipartite graphs.
    %  Finally, we generalize the results on Tur\'{a}n problems for Berge hypergraphs proposed by Gerbner, Methuku, and Palmer [\textit{Eur. J. Comb. 86 (2020) 103082}]. %are generalized.
\end{abstract}



\section{\normalsize Introduction}

% A \textit{hypergraph} $\cH=(V(\cH),E(\cH))$ consists of a vertex set $V(\cH)$ and a hyperedge set $E(\cH)$, where each hyperedge in $E(\cH)$ is a nonempty subset of $V(\cH)$. If $|e|=r$ for every $e\in E(\cH)$, then $\cH$ is called an \textit{$r$-uniform hypergraph} ($r$-graph for short). For simplicity, let $e(\cH):=|E(\cH)|$. The \textit{degree} $d_\cH(v)$ of a vertex $v$ is the number of hyperedges containing $v$ in $\cH$.

%In this paper, we use capital letter to represent graphs~(e.g., $F$, $G$), calligraphic capital letter to represent hypergraphs~(e.g., $\cG$, $\cH$).
%  frank letter to represent class of hypergraphs~(e.g., $\mathfrak{F}$). For a graph $F$, we use $V(F)$ to denote its vertex set, $E(F)$ to denote its edge set, and $e(F)$ to denote the number of edges, i.e., $e(F) = |E(F)|$. Similarly, for a hypergraph $\cH$, we define $V(\cH)$, $E(\cH)$ and $e(\cH)$ in the same way.

A \textit{hypergraph} $\cH$, denoted by $\cH=(V(\cH),E(\cH))$, on a finite vertex set $V(\cH)$ is a family $E(\mathcal{H})$ of subsets of $V(\cH)$, called \textit{hyperedges}.
For an integer $r\geq2$, a hypergraph $\cH$ is called an \textit{$r$-uniform hypergraph} ($r$-graph for short) if every hyperedge of $\cH$ contains exactly $r$ vertices. %We identify a hypergraph $\cH$ with its hyperedge set $E(\cH)$ and denote its vertex set by $V(\cH)$.
The size of $E(\cH)$ is denoted by $e(\cH)$. The \textit{degree} of a vertex $v$ in $\cH$, denoted $d_{\cH}(v)$, is the number of hyperedges of $\cH$ that contain $v$.
When there is no ambiguity regarding $\cH$, we simply write $d(v)$ instead of $d_\cH(v)$.
Throughout this paper, %we use capital letter to represent graphs (e.g., $F$, $G$), and
we assume that all $r$-graphs are simple, i.e., no loops and no multiple edges (hyperedges).

Given a family $\cF$ of $r$-graphs, we say $\cH$ is \textit{$\cF$-free} if $\cH$ does not contain any member of $\cF$ as a subhypergraph. The \textit{Tur\'{a}n number} ${\rm{ex}}_r(n,\cF)$ of $\cF$ is the maximum number of hyperedges in an $\cF$-free $r$-graph on $n$ vertices. When $r=2$, we write ${\rm{ex}}(n,\cF)$ instead of ${\rm{ex}}_2(n,\cF)$.
Tur\'{a}n problems on graphs and hypergraphs are central topics in extremal combinatorics. A classical result in extremal graph theory is the Erd\H{o}s-Gallai theorem \cite{gallai1959maximal}, which determines the exact Tur\'{a}n number of matchings. 

It is natural to investigate the Tur\'{a}n problem for matchings in hypergraphs. For a graph $F$ and an $r$-graph $\cF$, we say $\cF$ is a \textit{Berge copy} of $F$~(a Berge-$F$ for short) if $V(F) \subseteq V(\cF)$ and there is a bijection $\phi: E(F) \rightarrow E(\cF)$ such that $e \subseteq \phi(e)$ for each $e \in E(F)$. The graph $F$ is called a \textit{core} of $\cF$. Observe that for a fixed graph $F$ there are many hypergraphs that are Berge copies of $F$. For convenience, we refer to this collection of hypergraphs as $\mathcal{B}F$. 
In 1989, Berge~\cite{berge1984hypergraphs} introduced the notion of a \textit{Berge cycle}. Gy\H{o}ri, Katona and Lemons~\cite{GYORI2016238} defined the notion of \textit{Berge paths} and generalized the Erd\H{o}s-Gallai theorem to Berge paths in 2016. Later, Gerbner and Palmer \cite{gerbner2017extremal} generalized the established notions of Berge cycle and Berge path to general graphs.
Tur\'{a}n problems on Berge hypergraphs have been extensively studied, yielding numerous related results.
For a short survey on Tur\'{a}n problems on Berge hypergraphs, one can refer to Subsection 5.2.2 in~\cite{gerbner2018extremal}. 
In particular, Khormali and Palmer~\cite{KHORMALI2022103506} completely determined the Tur\'{a}n number of Berge matchings for sufficiently large $n$.
Let $M_{s+1}$ denote a matching of size $s+1$, i.e., the graph consisting of $s+1$ independent edges. 

\begin{thm}[Khormali and Palmer~\cite{KHORMALI2022103506}]\label{thm: Khormali Palmer}
Fix integers $s \geq 1$ and $r \geq 2$. Then for $n$ large enough,
\begin{eqnarray*}
{\rm{ex}}_r(n,\mathcal{B}M_{s+1})=
\begin{cases}
s , & {\rm{if}}\ r \geq 2s + 1; \\
\binom{2s+1}{r}, & {\rm{if}}\ s+1 < r < 2s + 1; \\
n - s , & {\rm{if}}\ r = s+1; \\
\binom{s}{r-1} (n - s ) + \binom{s}{r}, & {\rm{if}}\ r \leq s.
\end{cases}
\end{eqnarray*}
\end{thm}

Note that if $n \le 2s+1$, it is trivial that ${\rm{ex}}_r(n,\mathcal{B}M_{s+1}) = \binom{n}{r}$. For $n\ge 2s+2$ and $r\geq3$, Kang, Ni and Shan \cite{KANG2022112901} determined the exact value of $\ex_r(n,\mathcal{B}M_{s+1})$ in the case when $r\leq s-1$ or $r\geq 2s+2$. However, the case $s\leq r\leq 2s+1$ remains open.

\begin{thm}[Kang, Ni and Shan~\cite{KANG2022112901}]\label{thm: Kang Ni Shan}
    Fix integers $s \ge 1$ and $r \ge 3$. For any $n \ge 2s+2$.
    \begin{eqnarray*}
    {\rm{ex}}_r(n,\mathcal{B}M_{s+1})=
    \begin{cases}
    \max\{ \binom{2s+1}{r}, \binom{s}{r-1} (n - s ) + \binom{s}{r}\}, & {\rm{if}}\ r \leq s - 1; \\
    s & {\rm{if}}\ r \ge 2s+2.
    \end{cases}
    \end{eqnarray*}
\end{thm}

The purpose of this paper is to solve the remaining case of the Tur\'{a}n number of Berge matchings as stated in Theorem~\ref{thm: exact berge}. 
For non-negative integers $a$ and $b$, let $\binom{a}{b}$ denote the binomial coefficient. In particular, we define $\binom{a}{b}=0$ if $a<b$. %Here for positive integers $a$ and $b$, define $\binom{a}{b}=0$ if $a<b$.
% when $s\leq r\leq 2s+1$.

\begin{thm}\label{thm: exact berge}
    Let integers $s \ge 1$ and $r \ge 2$. For all $n \ge 2s+2$,
    \begin{eqnarray*}
    {\rm{ex}}_r(n,\mathcal{B}M_{s+1})=
    \begin{cases}
    \max\left\{ \binom{2s+1}{r}, \binom{s}{r-1} (n - s ) + \binom{s}{r}\right\}, & {\rm{if}}\ r \leq s + 1; \\
    \binom{2s+1}{r} , & {\rm{if}}\ s + 2 \leq r \leq 2s; \\
    s, & {\rm{if}}\ r \ge 2s+1.
    \end{cases}
    \end{eqnarray*}
\end{thm}

It is worth mentioning that the Tur\'{a}n number and the generalized Tur\'{a}n number of a matching together with another graph have been widely studied~(see~\cite{ALON2024223,Gerbner2024matching, 2023arXiv230105625M},\cite{Zhu}).

The $r$-expansion of a graph $F$ is a special member of $\mathcal{B}F$, defined by enlarging each edge of $F$ with a set of $r-2$ new vertices such that all these new sets are disjoint.
When considering the Tur\'{a}n number of the expansion of a matching, one obtains the famous Erd\H{o}s Matching Conjecture, which has been solved for sufficiently large $n$~\cite{FRANKL20131068}.
% When considering the Tur\'{a}n number of the expansion of a matching, it is the famous Erd艖s' Matching Conjecture, which is solved when $n$ is sufficiently large~\cite{FRANKL20131068}.
There are also many result on the Tur\'{a}n number of the expansion of a matching together with the expansion of another graph.
The reader may refer to~\cite{GERBNER2025104155, 2025arXiv250704579W, 2025arXiv251121096Y, 2025arXiv251117000C, zhou2026linear}.
The above results reveal the importance of matchings in extremal combinatorics.

\section{\normalsize Proof}\label{sec: berge exact}

Given an $r$-graph $\cH$, the \textit{neighborhood} $N_\cH(v)$ of a vertex $v$ in $\cH$ is the set of all vertices adjacent to $v$ in $\cH$. For a set $S\subseteq V(\cH)$, the \textit{neighborhood} of $S$ is defined as $N_\cH(S)=\bigcup_{v \in S} N_\cH(v)\setminus S$.
Recall that $d_{\cH}(v)$ denotes the \textit{degree} of a vertex $v$ in $\cH$.
%When $S$ contains only one vertex $v$, we simply write $N_\cH(v)$.
%And we use $d_\mathcal{H}(v)$ to denote the degree of $v$ in $\mathcal{H}$, i.e., $d_\mathcal{H}(v) = |N_\mathcal{H}(v)|$.
Let $\cH[U]$ denote the subhypergraph of $\cH$ induced by $U\subseteq V(\cH)$.
We say that a Berge matching $\cM$ in $\cH$ is \textit{maximum} if $\cH$ contains no Berge matching with a larger size than $\cM$.
When we say that $\cH$ contains a $\mathcal{B}M_{s+1}$, we mean that $\cH$ contains a Berge copy of $M_{s+1}$.
% When referring to $\cH$ contains a $\cB M_{s+1}$, we mean $\cH$ contains a Berge copy of $M_{s+1}$. %one of the hypergraphs in the family $\cB M_{s+1}$.

% Now let us start with the proof of Theorem~\ref{thm: exact berge} that we restate here for convenience.

% \begin{thm*}
%     Let integers $s \ge 1$ and $r \ge 2$. For all $n \ge 2s+2$,
%     \begin{eqnarray*}
%     {\rm{ex}}_r(n,\mathcal{B}M_{s+1})=
%     \begin{cases}
%     \max\left\{ \binom{2s+1}{r}, \binom{s}{r-1} (n - s ) + \binom{s}{r}\right\}, & {\rm{if}}\ r \leq s + 1; \\
%     \binom{2s+1}{r} , & {\rm{if}}\ s + 2 \leq r \leq 2s; \\
%     s, & {\rm{if}}\ r \ge 2s+1.
%     \end{cases}
%     \end{eqnarray*}
% \end{thm*}

Note that for the cases $r \le s-1$ and $r \ge 2s+2$ in Theorem \ref{thm: exact berge}, the result holds by Theorem~\ref{thm: Kang Ni Shan},
and the case  $r=2$ in Theorem \ref{thm: exact berge} follows from the
result of Erd\H{o}s and Gallai \cite{gallai1959maximal}. %Tutte-Berge formula.
Thus, in the following, we may assume $s \le r \le 2s+1$ and $r \ge 3$.

\begin{proof}[\bf Proof of Theorem~\ref{thm: exact berge}] 
    Let us first consider the lower bound.
    %We first consider the lower bounds by several extremal hypergraphs satisfying conditions.
    For $r \le s+1$, we consider the following two $r$-graphs:
    one is an $r$-graph formed by the union of a complete $r$-graph of order \(2s+1\) and \(n-2s-1\) isolated vertices (i.e., vertices of degree 0); the other is an $n$-vertex $r$-graph where there exists a vertex subset $S$ of size $s$ such that every hyperedge intersects $S$ in at least $r-1$ vertices.
    For $s+2 \leq r \leq 2s$, the lower bound follows from the union of a complete $r$-graph of order \(2s+1\) and \(n-2s-1\) isolated vertices. For $r\geq2s+1$, the lower bound follows from any $r$-graph with $s$ hyperedges.
    It is easy to verify that these hypergraphs are $\mathcal{B}M_{s+1}$-free.
        %For $r \geq 2s+1$,  $n$ vertices and arbitrary $s$ hyperedges is the extremal graph.

    In the following, we prove the upper bound.
    Let $\mathcal{H}$ be a $\cB M_{s+1}$-free $r$-graph on $n$ vertices with maximum number of hyperedges.
    Note that adding a new hyperedge to $\cH$ increases the size of the core of a maximum Berge matching by at most $1$.
    Thus, by the maximality of $e(\cH)$, we may assume that $\mathcal{H}$ contains a $\cB M_s$, denoted by $\cM$. Let $\cE=\{e_1,\dots,e_s\}$ be the hyperedges of $\cM$, with the core $\cC = \{u_1,v_1,\dots,u_s,v_s\}$ such that $\{u_i,v_i\}\subseteq e_i$ for each $i\in [s]$.
    For a subset $S \subseteq \cC$, we write $\overline{S} = \{u_i \mid v_i \in S\} \cup \{v_i \mid u_i \in S\}$.
   % Note that for every $e \in E(\cH)\setminus \cE$, $|e\cap \cC|\geq s-1$, otherwise we can find a Berge copy of $M_{s+1}$ by adding the hyperedge $e$ to $\cE$.
    %For a vertex $v \notin \cC$, let $E_{\cC}(v) = \{e \in E(\cH)\setminus \cE \mid v \in e\}$, $d_{\cC}(v) = |E_{\cC}(v)|$.
    %Note that $d_{\cC}(v)$ is not exactly the degree of $v$ in $\cH$, since we do not count the hyperedges in $\cE$.
    %Let $N_\cC(v) = \{u \mid \exists e \in E_\cC(v), u \in e\}$ be the neighbors of $v$.
   % Note that $N_\cC(v) \subseteq \cC$, and $|N_\cC(v)| = r-1$ if and only if $d_\cC(v) = 1$.
   % Note that $|S| = |\overline{S}|$.
    % Set $d_{\cC}(v)=|N_{\cC}(v)|$. For every $v\in V(\cH)\setminus {\cC}$, let $C(v)=\{E\in N_{\cC}(v):~\{u_i,v_i\}\subseteq E \text{~for some~} i\in [k]\}$. And for every $v\in V(\cH)$, let $D(v)=\{u\in V(\cH):~\text{there exists $E\in E(\cH)$ such that $\{u,v\}\subseteq E$}\}$ be the neighbourhoods of $v$.

    Let $\cH'$ be the $r$-graph obtained from $\cH$ by deleting $s$ hyperedges of $\cE$. %Let $\mathcal{H}'=\mathcal{H}-\cE$.
    Let $E_{\cH'}(v)$ denote the set of hyperedges in $\cH'$ that contain $v$.
    Observe that each hyperedge of $\cH'$ intersects $\cC$ in at least $r-1$ vertices. Indeed, if there exists a hyperedge $e\in E(\cH')$ such that $e\setminus \cC=\{u,v\}$, then $\cE\cup \{e\}$ forms a $\cB M_{s+1}$ in $\cH$ with the core $\cC\cup\{u,v\}$, a contradiction.
    This implies that $N_{\cH'}(v) \subseteq \cC$ for any $v\in V(\cH)\setminus \cC$.
    %$u\in V(\cH)\setminus \cC$ such that $\{u,v\}\subseteq e\in E(\cH'then $\cE\cup \{e\}$ is a Berge copy of $M_{s+1}$  with the core $\cC\cup \{v,u\}$ of $\cH$, a contradiction.

    %The following claim holds for the case when $s\leq r\leq 2s$.
    \begin{cla}\label{claim: 2 isolated vtx}
        If $s\leq r\leq 2s$,
        then either $e(\cH)\leq \binom{2s+1}{r}$,
        or there are at least two vertices $u, v\in V(\cH)\setminus \cC$ such that $d_{\cH'}(u) \geq 1$ and $d_{\cH'}(v) \geq 1$.
    \end{cla}

    \begin{proof}[{\bf{Proof of Claim~\ref{claim: 2 isolated vtx}.}}]
    %\noindent\textbf{Proof of Claim~\ref{claim: 2 isolated vtx}:}
    It is sufficient to prove that if there is at most one vertex $u\in V(\cH)\setminus \cC$ with $d_{\cH'}(u) \geq 1$, then $e(\cH)\leq \binom{2s+1}{r}$.
    If $d_{\cH'}(v)=0$ for every $v\in V(\cH)\setminus \cC$, then $e(\cH) \leq s+\binom{2s}{r}\leq \binom{2s+1}{r}$ as $r\leq 2s$.

    Let $v\in V(\cH)\setminus \cC$ be the unique vertex with $d_{\cH'}(v)\geq 1$. If $e_i \subseteq \cC \cup \{v\}$ for every $i\in [s]$, then $e(\cH) \leq \binom{2s+1}{r}$ and we are done. Now suppose that there exists some $i\in [s]$ such that $e_i \nsubseteq \cC \cup \{v\}$.
    Without loss of generality, suppose $e_s$ contains a vertex $w \notin \cC \cup \{v\}$. Note that $|\cC|=2s$. If $r=2s$, then we have $u_s\in N_{\cH'}(v)$ or $v_s\in N_{\cH'}(v)$. Thus, there exists a $\cB M_{s+1}$ in $\cH$ by adding the hyperedge containing $v$ and $u_s$~(or $v_s$) to $\cE$, and extending the core to $\cC\cup \{v,w\}$, a contradiction. Now we assume that $s\leq r<2s$.
    %Note that $N_{\cH'}(v) \subseteq \cC$. Indeed, if there exists a vertex $u\in V(\cH)\setminus \cC$ such that $\{u,v\}\subseteq e\in E(\cH')$, then $\cE\cup \{e\}$ is a Berge copy of $M_{s+1}$  with the core $\cC\cup \{v,u\}$ of $\cH$, a contradiction.
    %Next we consider two cases: one is $e_i \subseteq \cC \cup \{v\}$ for every $i\in [s]$, and the other is that there exists $e_i \nsubseteq \cC \cup \{v\}$ for some $i\in [s]$.
    %For the first case, we have  $e(\cH) \leq \binom{2s+1}{r}$.
    We claim that $u_s\notin N_{\cH'}(v)$ and $v_s \notin N_{\cH'}(v)$.
    Indeed, if $u_s\in N_{\cH'}(v)$ (or $v_s\in N_{\cH'}(v)$), then as in the case $r=2s$ we may find a $\cB M_{s+1}$ in $\cH$, a contradiction.
    %If $r=2s$, then we have $u_s\in N_{\cH'}(v)$ or $v_s\in N_{\cH'}(v)$. So there exists a $\cB M_{s+1}$ in $\cH$, a contradiction.

    Then $|E_{\cH'}(v)|\leq \binom{|N_{\cH'}(v)|}{r-1}$.
    Note that $E(\cH'[\cC])=E(\cH')\setminus E_{\cH'}(v)$.
    %Note that every hyperedge of $E(\cH')\setminus E_{\cH'}(v)$ is contained in the vertex set $\cC$.
    We now partition $E(\cH'[\cC])$ into the following three classes:
    $\cE_1$ (consisting of hyperedges not containing $\overline{N_{\cH'}(v)}$),
    $\cE_2$ (consisting of hyperedges not containing $\{u_s,v_s\}$),
    and $\cE_3$ (consisting of hyperedges containing both a vertex of $\{u_s,v_s\}$ and a vertex of $\overline{N_{\cH'}(v)}$).
    %We now consider  the edges of $E(\cH')\setminus E_{\cH'}(v)$ formed by $\cE_1$ (hyperedges  not containing $\overline{N_{\cH'}(v)}$), $\cE_2$ (hyperedges not containing $\{u_s,v_s\}$),  and $\cE_3$ (hyperedges containing both  one vertex in $\{u_s,v_s\}$ and one  in $\overline{N_{\cH'}(v)}$).
    We assert that $\cE_3=\emptyset$. Otherwise, assume that $e\in \cE_3$ contains the vertex $u_s$ and $u_i\in \overline{N_{\cH'}(v)}$ (or $v_i\in \overline{N_{\cH'}(v)}$) for some $i\in [s-1]$.
    By the definition of $\overline{N_{\cH'}(v)}$, there is a hyperedge $e'\in E(\cH')$ containing $v$ and $v_i$ (or $u_i$). Then $(\cE\setminus \{e_i\})\cup \{e',e\}$ forms a $\cB M_{s+1}$ in $\cH$ with the core $\cC\cup \{v,w\}$, a contradiction.
    %i.e., for any hyperedge $e\in E(\cH')\setminus  E_{\cH'}(v)$, $e$ can not contain both one vertex in $\{u_s,v_s\}$ and one  in $\overline{N_{\cH'}(v)}$.
    %Suppose, without loss of generality, that $e\in \cE_3$ contains the vertex $u_s$ and $u_i\in \overline{N_{\cH'}(v)}$ for some $i\in [s-1]$.
    Note that $r-1\leq|N_{\cH'}(v)|\leq 2s-2$. Therefore, we have
     \begin{equation*}
    \begin{aligned}
        e(\cH)= s+|E_{\cH'}(v)|+|\cE_1|+|\cE_2| &\leq s + \binom{|N_{\cH'}(v)|}{r-1} + \binom{2s-|N_{{\cH'}}(v)|}{r} + \binom{2s-2}{r}
        %&\leq s + \binom{2s-2}{r-1} + \binom{2s-r+1}{r}+\binom{2s-2}{r} \\
        \le \binom{2s+1}{r}.
    \end{aligned}
     \end{equation*}
     \end{proof}
     %\hfill $\square$ \par

    %  Then $D(v)\subseteq U$. If for every $E_i$, $E_i\subseteq U$, then $|E(\cH)|\leq \binom{2k+1}{r}$.
    %  If there is some $E_i$, say $E_k$, with $E_k\setminus U\neq \emptyset$.
    %  Then for every $u_i\in D(v)$, there is no $E\in E(\cH)$ such that  $\{v_i,u_k\}\subseteq E$ or $\{v_i,v_k\}\subseteq E$. Otherwise, suppose $E'\in E(\cH)\setminus M$ with $\{v,u_i\}\subseteq E'$, then $(M\setminus E_i)\cup (E'\cup E_k)$ contains a copy of $\cB M_{k+1}$.
    %  Since $D(v)\geq k-1$, the number of hyperedges besides $M$ intersecting with $\{u_k,v_k\}$ is at most $\binom{|U\setminus D(v)|}{r}\leq k+1$.
    %  And the number of hyperedges besides $M$ do not intersect with $\{u_k,v_l\}$ is at most $\binom{|U\setminus\{u_k,v_k\}\cup \{v\}|}{k}=\binom{2k-1}{r}$. Thus, we have
    %  \begin{align*}
    %      |E(\cH)|\leq & |\{E\in E(\cH)\setminus M:~E\cap \{u_k,v_k\}\neq \emptyset\}|+|\{E\in E(\cH)\setminus M:~E\cap \{u_k,v_k\}= \emptyset\}|+|M|\\
    %      \leq & k+1+\binom{2k-1}{r}+k\leq \binom{2k+1}{r}
    %  \end{align*}
    %  When $k\leq r\leq 2k$.
    %  \hfill $\square$

    % \noindent\textbf{Situation (i): $r=s$.}
    \subsection{\normalsize Proof of Theorem~\ref{thm: exact berge} for $r=s$}

    \begin{cla}\label{claim: outdegree at least 2}
        If $d_{\cH'}(v)\leq 1$ for all $v\in V(\cH)\setminus \cC$,  then $e(\cH)\leq \max\left\{\binom{2s+1}{s},s(n-s)+1\right\}$.
    \end{cla}
    \begin{proof}[{\bf Proof of Claim \ref{claim: outdegree at least 2}.}]
    %\noindent\textbf{Proof of Claim~\ref{claim: outdegree at least 2}:}
    Note that $e(\cH'[\cC])\leq \binom{2s}{r}=\binom{2s}{s}$. Since $d_{\cH'}(v)\leq 1$ for all $v\in V(\cH)\setminus \cC$,
    \begin{align*}
        e(\cH) &= s+e(\cH'[\cC])+\sum_{v\in V(\cH)\setminus \cC}d_{\cH'}(v)\leq s+\binom{2s}{s}+n-2s=\binom{2s}{s}+n-s.
        %&\le \max \left\{\binom{2s+1}{s},1+s(n-s)\right\}~~~~(\text{since}~s\ge 3~\text{and}~n\ge 2s+2).
    \end{align*}
    If $e(\cH)>\max\left\{\binom{2s+1}{s},s(n-s)+1\right\}$, then we have $\binom{2s}{s}+n-s\geq \binom{2s+1}{s}+1$ and $\binom{2s}{s}+n-s\geq s(n-s)+2$. This implies that $\binom{2s}{s}\geq (s-1)(n-s)+2\geq (s-1)\big[\binom{2s}{s-1}+1\big]+2=(s-1)\binom{2s}{s-1}+s+1$, which is a contradiction when $s\geq2$.
    \end{proof}
    %\hfill $\square$

  % By Claim~\ref{claim: outdegree at least 2},  we may assume that  there exist some $v\in V(\cH) \setminus \cC$ with $d_{\cH'}(v) \ge 2$.

    \smallskip
    \noindent\textbf{Case 1.}  There exists $w \in V(\cH)\setminus \cC$ such that $\{u_i,v_i\} \subseteq N_{\cH'}(w)$ for some $i\in [s]$.
    \smallskip
    \smallskip

   %     In other words, one of the following holds:
   %     \begin{enumerate}
    %        \item  $\exists e_w \in E_{\cC}(w)$ such that $u_i,v_i \in e_w$.
    %        \item  $\exists e_{w1}, e_{w2} \in E_{\cC}(w)$ such that $u_i \in e_{w1}, v_i \in e_{w2}$.
    %    \end{enumerate}
    % We choose a such vertex $w$ such that $d_{\cH'}(w)$ achieves minimum.
    We choose such a vertex $w$ so that $d_{\cH'}(w)$ is minimal.
    Without loss of generality, assume that $\{u_1,v_1\} \subseteq N_{\cH'}(w)$. Then we have $\{u_1,v_1\} \subseteq e_w\in E_{\cH'}(w)$, or $u_1\in e_{w}^1$ and $v_1\in e_{w}^2$, where $e_{w}^1,e_{w}^2\in E_{\cH'}(w)$ and $e_{w}^1\neq e_{w}^2$.
        Let
        \[
        Y = N_{\cH'}\big(V(\cH)\setminus \left(\cC \cup\{w\}\right)\big).
        \]
    Clearly, $Y \subseteq \cC$. We claim that $u_1,v_1\notin Y$. Otherwise, there exists a $\cB M_{s+1}$ in $\cH$ with the core $\cC\cup\{w,z\}$, where $z\in V(\cH)\setminus \left(\cC \cup\{w\}\right)$ and $u_1\in N_{\cH'}(z)$ (or $v_1\in N_{\cH'}(z)$).

    By Claim~\ref{claim: 2 isolated vtx}, we may suppose that there exists at least one vertex in $V(\cH)\setminus \left(\cC \cup\{w\}\right)$ such that its degree in $\cH'$ is at least 1. Otherwise, $e(\cH)\leq \binom{2s+1}{r}$ and we are done.
    Thus, $|Y|\ge r-1=s-1$. %We next discuss two cases.

    \smallskip
    \noindent\textbf{Case 1.1.} $|Y| = s-1$.
    \smallskip
    \smallskip

    It is easy to see that $d_{\cH'}(v)\le 1$ for any $v\in V(\cH)\setminus \{\cC\cup w\}$. By Claim~\ref{claim: outdegree at least 2}, we may suppose $d_{\cH'}(w) \ge 2$. Otherwise, $e(\cH)\leq \max\left\{\binom{2s+1}{s},s(n-s)+1\right\}$ and we are done.
    Let $\cE_{1}$ and $\cE_{2}$ denote the set of hyperedges in $\cH'$ containing a vertex in $V(\cH)\setminus (\cC \cup \{w\})$, and the vertex $w$, respectively. Note that $\cE_{2}=E_{\cH'}(w)$. Let $\Omega = \{v_1,\ldots,v_s\}$ and $\cE_{3}=E(\cH'[\cC])\backslash E(\cH'[\Omega])$. Then
  \begin{align*}
  E(\cH)= \cE\cup E(\cH'[\Omega])\cup \bigcup_{i=1}^{3}\cE_{i}.
  \end{align*}
Clearly, $e(\cH'[\Omega])\le 1$. Note that $d_{\cH'}(v) \le 1$ for any $v \in V(\cH) \setminus (\cC \cup \{w\})$.
So $|\cE_{1}|\le n-2s-1$.
%It is sufficient to dic   the size of $\cE_2$ and $\cE_3$.

    Recall that $u_1,v_1\notin Y$. We claim that  $\{u_i,v_i\}\nsubseteq Y$ for all $2\leq i\leq s$. Otherwise, there exists some $j$ such that $\{u_j,v_j\}\subseteq Y$. Then $V(\cH) \setminus (\cC \cup \{w\})$  contains a vertex $v$ such that $\{u_j,v_j\} \subseteq N_{\cH'}(v)$ and $d_{\cH'}(v)\le 1$,  which contradicts the minimality of $d_{\cH'}(w)$.
    Without loss of generality, we assume that $Y = \{v_2,\ldots,v_s\}$.
    %Since $d_{\cH'}(w) \ge 2$, we have $|N_{\cH'}(w)| \ge r=s$.
    We claim that $N_{\cH'}(w) \subseteq \Omega \cup \{u_1\}$. Indeed, if $u_i\in N_{\cH'}(w)$ for some $2\le i \le s$, then there exists a $\cB M_{s+1}$, obtained by adding to $\cE\setminus \{e_i\}$ the hyperedge containing $wu_i$ and the hyperedge containing $v_iw'$ for some $w' \in V(\mathcal{H}) \setminus (\cC \cup \{w\})$ by the definition of $Y$.
    So $|\cE_2|\le \binom{s+1}{s-1}=\binom{s+1}{2}$.

    Since $d_{\cH'}(w) \ge 2$, we have $|N_{\cH'}(w)| \ge r=s$.
    Recall that $\{u_1,v_1\} \subseteq N_{\cH'}(w)$.
    Without loss of generality, we may assume that $\{u_1,v_1,\ldots,v_{s-1}\} \subseteq N_{\cH'}(w)$.
    For any hyperedge $e \in E(\cH'[\cC])$ containing $u_1$, $e$ does not contain any $u_j$ for $2\leq j\leq s$. Otherwise, by the definition of $Y$, we may find a $\cB M_{s+1}$ in $\cH$ whose core has edge set $\{wv_1,u_1u_j,v_jw',u_2v_2,\dots,u_sv_s\}\backslash\{u_jv_j\}$.
    For any hyperedge $e \in E(\cH'[\cC])$ containing $u_i$ with $2\leq i\leq s$, $e$ does not contain any other $u_j$ ($1\leq j\leq s$ and $j\neq i$) or $v_1$. Otherwise, by the definition of $Y$, we may find a $\cB M_{s+1}$ in $\cH$ (see Figure~\ref{fig: case 1.1}).
    It means that $|\cE_3|\le \binom{s}{r-1}+s-1=2s-1$.
    %For any hyperedge $e \in E(\cH'[\cC])$ containing $u_i$, $e$ does not contain any $u_j$ ($2\leq j\leq s$) or $v_1$. Otherwise, we may find a $\cB M_{s+1}$ in $\cH$~(see Figure~\ref{fig: case 1.1}). It means $|\cE_3|\le s$.

        \begin{figure}
            \centering
            \includegraphics[width=0.9\linewidth]{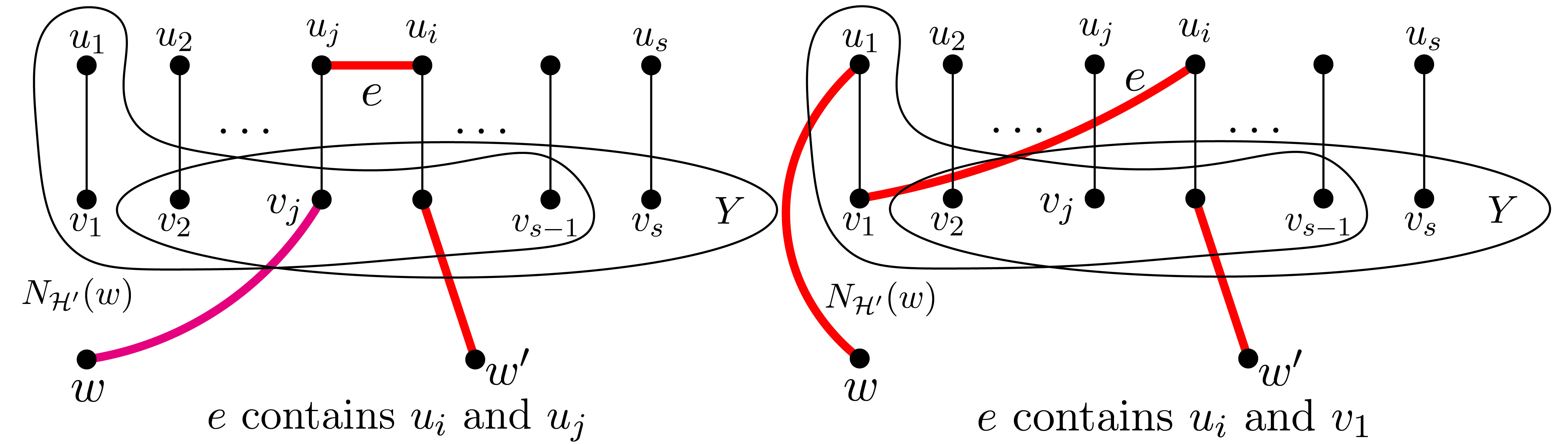}
            \caption{The illustration of $\cB M_{s+1}$ by the thick red lines. The left figure shows when a hyperedge $e \in \cE_3$ contains $u_i, u_j$~(the figure shows the case $2 \le i,j \le s-1$, the case when $i \in \{1,s\}$ holds similarly).
            The right figure shows when a hyperedge $e\in \cE_3$ contains $u_i, v_1$ ($2\leq i\leq s$).
            The hyperedge connecting $w'$ comes from the definition of $Y$.
            }
            \label{fig: case 1.1}
        \end{figure}

        Therefore, we have
        \begin{equation*}
        \begin{aligned}
            e(\cH)\leq   |\cE|+ e(\cH'[\Omega])+\sum_{i=1}^{3}|\cE_{i}| &\le s + 1+(n - 2s - 1) + \binom{s+1}{2} + 2s-1 \\
            & = n+s-1+\binom{s+1}{2} \\
            &\le \max\left\{ \binom{2s+1}{s}, s(n-s)+1 \right\},
        \end{aligned}
        \end{equation*}
        where the last inequality can be proved by a method similar to that of Claim \ref{claim: outdegree at least 2} when $s\geq3$.

    \smallskip
    \noindent\textbf{Case 1.2.} $|Y| \ge s$.
    \smallskip
    \smallskip
        %Now we consider $e_1 \in \cE$ containing $u_1v_1$.

    Let $\cE_{1}$ and $\cE_{2}$ denote the set of  hyperedges in $\cH'$ containing the vertex $w$, and a vertex in $V(\cH)\setminus (\cC \cup \{w\})$, respectively. Let $\cE_{3}=E(\cH'[\cC])$.
    %and the some vertices in $\cC$, and
    Then
    \begin{align*}
    E(\cH)= \cE\cup \bigcup_{i=1}^{3}\cE_{i}.
    \end{align*}

    Recall that $\{u_1,v_1\}\subseteq N_{\cH'}(w)$. %and
    We claim that $e_1 \subseteq \cC\cup \{w\}$. Otherwise, we may find a $\cB M_{s+1}$ in $\cH$ with the core $\cC\cup\{w,z\}$, where $z\in e_1\setminus \left(\cC\cup \{w\}\right)$.
    So we can choose a set $X$ of $s$ vertices from $(e_1 \cup e_w) \cap \cC$ or $(e_1 \cup e_{w}^1 \cup e_{w}^2) \cap \cC$ such that  $\{u_1,v_1\}\subseteq X$.

    Recall that $u_1,v_1\notin Y$. It follows that $u_1,v_1\notin \overline{Y}$ and $|\overline{X} \cap \overline{Y}| \le s-2$.
    We claim that $\overline{X} \cap \overline{Y} = \emptyset$.
    Otherwise, assume that $u_j\in \overline{X} \cap \overline{Y}$ (or $v_j\in \overline{X} \cap \overline{Y}$)  for some $2 \le j \le s$.
    Then for the hyperedge $e_j \in \cE$, apart from $u_j$~(or $v_j$), it does not contain vertices of $V(\cH)\setminus \cC$, vertices in $\overline{X}$, and vertices in $\overline{Y}$. Otherwise, we may obtain a $\cB M_{s+1}$ in $\cH$~(see Figure~\ref{fig: Berge exact} for illustration), a contradiction.
    Thus, we have
        \begin{equation*}
        \begin{aligned}
            s- 1=|e_j| -1 \le |\cC \setminus (\overline{X} \cup \overline{Y})| = 2s-|X|-|Y| + |\overline{X} \cap \overline{Y}| \le s - |Y| + s - 2.
        \end{aligned}
        \end{equation*}
        It yields $|Y| \le s-1$, which contradicts $|Y|\ge s$. Thus, we have $\overline{X} \cap \overline{Y} = \emptyset$.

        \begin{figure}[t]
            \centering
            \includegraphics[width=0.9\textwidth]{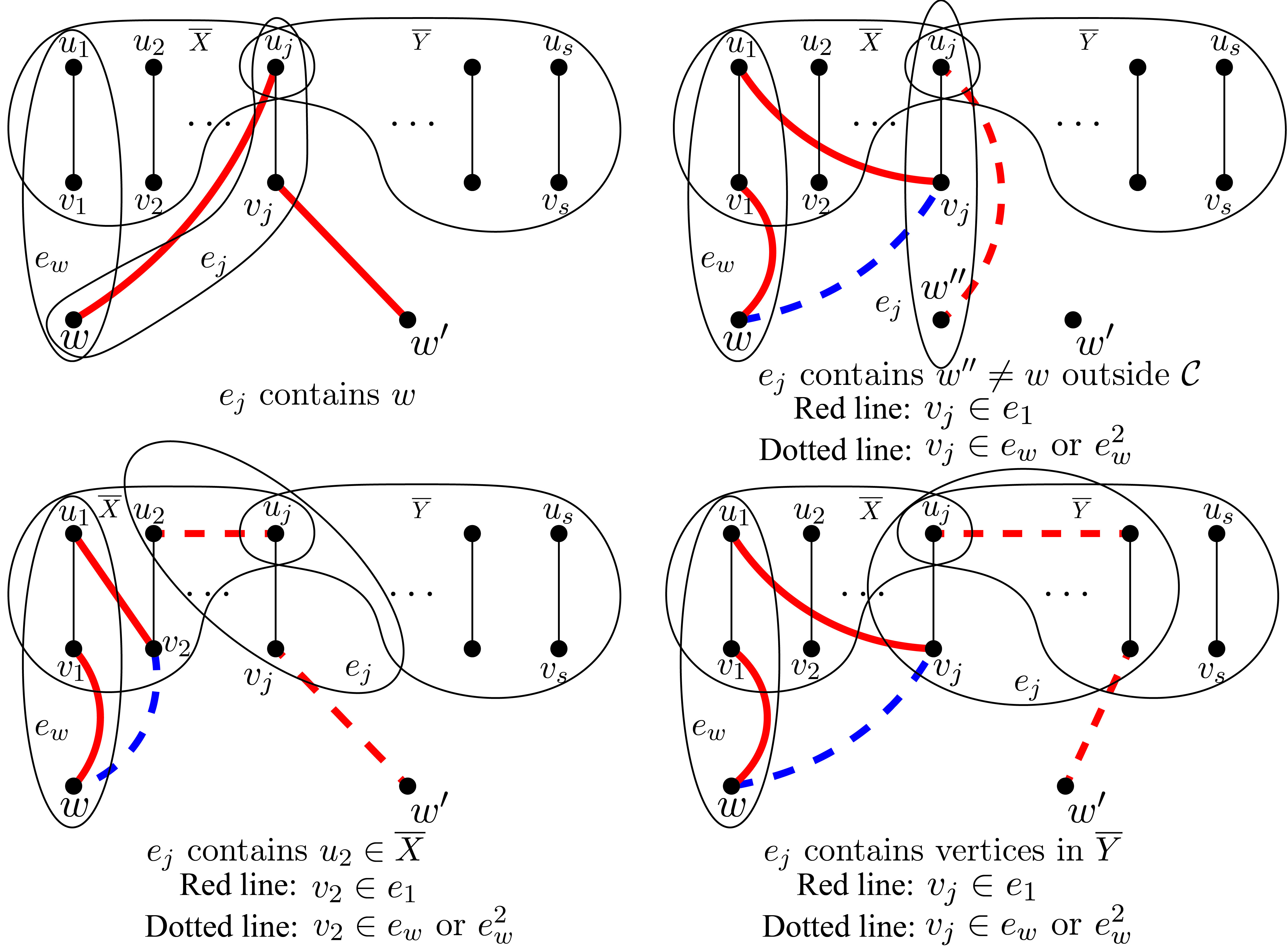}
            \caption{The four cases: $e_j$ contains $w$, $w'' \neq w$, a vertex in $\overline{X}$, or a vertex in $\overline{Y}$. In each case, we can find a $\cB M_{s+1}$ using either red lines or dotted lines.
            In the picture, we only list some typical cases. For example, when $e_j$ contains $u_2 \in \overline{X}$, we only consider $u_2 \in e_j\cap \overline{X}$.
            However, other cases hold similarly.}\label{fig: Berge exact}
        \end{figure}

        Combining $|X|=|\overline{X}| = s$, $|Y|=|\overline{Y}| \ge s$, $|\overline{X} \cup \overline{Y}| \le |\cC|= 2s$, and  $|\overline{X} \cap \overline{Y}| = 0$, we have $|Y| = |\overline{Y}| = s$ and $|X \cap Y |= 0$.
        We claim that $X \cap \overline{Y} = \emptyset$. Otherwise, there exists a $\cB M_{s+1}$ in $\cH$. So we have $X = \overline{X}$ and $Y = \overline{Y}$.
        Thus, all hyperedges of $\cE_1$ are contained in $\{w\} \cup X$ (otherwise, there exists a $\cB M_{s+1}$ in $\cH$ since $Y = \overline{Y}$).
        Then $|\cE_1|\leq \binom{|X|}{s-1}=\binom{s}{s-1}$.
        Furthermore, if $e\in\cE_3$ contains a vertex from $X$, then $e=X$, and if $e\in\cE_3$ contains a vertex from $Y$, then $e=Y$ (otherwise, there exists a $\cB M_{s+1}$ in $\cH$).
        Then $|\cE_3|\leq 2$.
        By the definition of $Y$, we have $|\cE_2|\leq \binom{|Y|}{s-1}(n-2s-1)=\binom{s}{s-1}(n-2s-1)$.
        %all hyperedges of $\cE_2$ are contained in $\{w'\}\cup Y$, and the hyperedges of $\cE_3$ either contains $u_i \in X$ or $v_i \in X$, or contains $u_i \in Y$ or $v_i \in Y$.
        Therefore,
        \begin{equation*}
        \begin{aligned}
            e(\cH) &= |\cE|+\sum_{i=1}^{3}|\cE_{i}|\\
            & \leq s + \binom{s}{s-1}  +\binom{s}{s-1}(n - 2s - 1) + 2 \\
            & = s(n - s) + 1 - (s^2 -s -1) \le s(n-s)+1.
        \end{aligned}
        \end{equation*}

    \smallskip
    \noindent\textbf{Case 2.} For any $v \in V(\cH)\setminus \cC$, $\{u_i,v_i\} \nsubseteq N_{\cH'}(v)$ for each $i\in [s]$.
    \smallskip
    \smallskip

    By Claim~\ref{claim: outdegree at least 2}, we may suppose $d_{\cH'}(w)\ge 2$ for some $w\in V(\cH)\setminus \cC$. Then we have that $N_{\cH'}(w) \subseteq \cC$ and $|N_{\cH'}(w)| \ge r=s$.
    Note that $\{u_i,v_i\}\nsubseteq N_{\cH'}(w)$ for every $i\in[s]$.
    Hence, without loss of generality, we may assume that $N_{\cH'}(w) = \Omega$.

    By Claim~\ref{claim: 2 isolated vtx}, we may suppose that there exists at least one vertex in $V(\cH)\setminus \left(\cC \cup\{w\}\right)$ such that its degree in $\cH'$ is at least 1. Otherwise, $e(\cH)\leq \binom{2s+1}{r}$ and we are done.
    It follows that $N_{\cH'}(w') \subseteq \Omega$ for any $w' \in V(\cH)\setminus \left(\cC \cup\{w\}\right)$. Otherwise, there exists a $\cB M_{s+1}$ in $\cH$ with the core $\cC\cup\{w,w'\}$.

    % No edge in $E(\cH)\setminus \cE$ contains $w \in V(\cH)\setminus \cC$ and $u_i,v_i$ for some $i$.

    % By the special properties of $v_i$, $N_{\cC}(w) \subseteq \Omega = \{v_1,\ldots,v_s\}$ for all $w\in V(\cH)\setminus\cC$.

    % \begin{cla}\label{claim: cover once}
    %     If there exists a set $W \subseteq V(\cH)\setminus \cC$ such that each $v_i$, $i=1,2,\ldots,s$ is in at least one set in ${\{N_{\cC}(w)\}}_{w\in W}$.
    % Then every $e_i$ can not contain any vertex outside $\cC \cup W$.
    % Also, if $e$ is an edge from $E(\cH) \setminus \cE$, if $e$ contains some $u_i$, then $e \subseteq \cC \cup W$.
    % \end{cla}

    \smallskip
    \noindent\textbf{Case 2.1.}
    There exists a set $W \subseteq V(\cH)\setminus \cC$ such that each $v_i$ $(i\in [s])$ is contained in at least two members of the family $\{N_{\cH'}(v)\mid	v\in W\}$.
    \smallskip
    \smallskip

    It follows that for each $e \in E(\mathcal{H})$, the hyperedge $e$ can not contain both $u_i$ and $u_j$ for any $i \neq j$.
    Otherwise, we would get a $\cB M_{s+1}$ in $\cH$ by connecting $u_iu_j$ with the edge, and connecting $u_i,u_j$ to two different elements in $W$ by definition.
    Then every hyperedge $e \in E(\mathcal{H})$ intersects $\Omega$ in at least $s-1$ vertices.
    % Similarly, for any hyperedge $e\in E(\cH')$, $e$ intersects $\Omega$ in at least $s-1$ vertices.
    Therefore, we have
    \begin{equation*}
    \begin{aligned}
        e(\cH) & \leq 1 + \binom{s}{s-1}(n-s) = 1 +s(n-s).
    \end{aligned}
    \end{equation*}

    \smallskip
    \noindent\textbf{Case 2.2.} There does not exist $W \subseteq V(\cH)\setminus \cC$ such that each $v_i$ $(i\in [s])$ is contained in at least two members of the family $\{N_{\cH'}(v)\mid v\in W\}$. %For any $W \subseteq V(\cH)\setminus \cC$, there exist some $v_i$ are  in at most  one set of the family $\{N_{\cH'}(v)\mid	v\in W\}$, where  $i\in [s]$.
    \smallskip
    \smallskip

     % Let $\cE_{1}$,  $\cE_{2}$,  $ \cE_{3}$, and $ \cE_{4}$ denote the set of  hyperedges   containing  $w$, $w' \in V(\cH)\setminus (\cC \cup \{w\})$, $u_1$, and  some vertices in $\cC\setminus \{u_1\}$, respectively.
%  Then
 % \begin{align*}
 % E(\cH)= \cE\cup \bigcup_{i=1}^{4}\cE_{i}.
 % \end{align*}

    %Recall that there exists at least one vertex in $V(\cH)\setminus \left(\cC \cup\{w\}\right)$ such that its degree in $\cH'$ is at least 1.
    Note that $N_{\cH'}(w) = \Omega$ and $N_{\cH'}(w') \subseteq \Omega$ for any $w' \in V(\cH)\setminus (\cC \cup \{w\})$. It follows that for every $w' \in V(\cH)\setminus (\cC \cup \{w\})$, we have $d_{\cH'}(w') \le 1$.
    Otherwise, there exists a set $W=\{w,w'\}$ that satisfies Case~2.1, a contradiction.
    So, by Claim \ref{claim: 2 isolated vtx}, there exists one vertex $w'\in V(\cH)\setminus \left(\cC \cup\{w\}\right)$ such that $d_{\cH'}(w')=1$.
    Similarly, if $w_1',w_2' \in V(\cH)\setminus (\cC \cup \{w\})$ satisfies $d_{\cH'}(w_1') = d_{\cH'}(w_2') = 1$,
    then $N_{\cH'}(w_1') = N_{\cH'}(w_2')$.
    %Otherwise, we reduce to Case 2.1.
    Without loss of generality, assume that $N_{\cH'}(w') = \Omega\setminus \{v_1\}$ for all $w' \in V(\cH)\setminus (\cC \cup \{w\})$ with $d_{\cH'}(w') = 1$.
        Let $\cE_{1}$ and $\cE_{2}$ denote the set of  hyperedges in $\cH'$ containing the vertex $w$, and a vertex in $V(\cH)\setminus (\cC \cup \{w\})$, respectively. Let $\cE_{3}=E(\cH'[\cC])$.
Clearly, we have $|\cE_1| \le \binom{s}{s-1} = s $, $|\cE_2| \le n-2s-1$.

Now we consider the hyperedges in $\cE_{3}$.
If $e \in \cE_{3}$ contains $u_i$~($1 \le i \le s$), then it must lie in $\{u_i\}\cup \Omega$, at most $\binom{s}{s-1} = s$ such hyperedges. Otherwise, we would get a $\cB M_{s+1}$ in $\cH$ by connecting $u_iu_j$ with the edge, and connecting $u_i,u_j$ to $w,w'$.
%If $e \in \cE_{3}$ contains $u_i$~($2 \le i \le s$), then it must be the hyperedge $\{u_i\} \cup \Omega \setminus \{v_1\}$, at most $s-1$ such hyperedges.
With only one hyperedge on $\Omega$, we have $|\cE_3| \le s^2+1$.

    % Therefore, the number of hyperedges in $\cH'$ containing a vertex in $V(\cH)\setminus (\cC \cup \{w\})$ is at most $n-2s-1$.

    % the number of hyperedges in $E(\cH'[\cC])$ containing $u_1$...

    % the number of hyperedges in $E(\cH'[\cC])$ containing $u_i$...

    % Note that the number of hyperedges in $\cH'$ containing $w$ is at most $\binom{s}{s-1}$. Thus,

    %some $w' \in V(\cH)\setminus (\cC \cup \{w\})$, $u_1$, and some $u_i$ ($2 \le i \le s$)  are at most $\binom{s}{s-1}$,  $n - 2s - 1$, $\binom{s+1}{s-1}$, and $s-1$, respectively.
    This implies that
    \begin{equation*}
    \begin{aligned}
        e(\cH) & \le |\cE|+|\cE_1| + |\cE_2| + |\cE_3|\\
        % e(\cH) & \le |\cE|+e(\cH[\Omega])+\binom{s}{s-1}+(n - 2s - 1)+\binom{s+1}{s-1}+(s-1)\\
        &\leq s + s + n-2s-1 +s^2+1 \\
               & = n + s^2  \le \max\left\{ \binom{2s+1}{s}, 1+s(n-s) \right\},
    \end{aligned}
    \end{equation*}
    where the last inequality can be proved by a method similar to that of Claim \ref{claim: outdegree at least 2} when $s\geq2$.

    % then by Claim~\ref{claim: outdegree at least 2}, and there are two vertices $w_1,w_2 \in V(\cH)\setminus \cC$ with $d_{\cC}(w_i) \geq 1$ for $i=1,2$ and $N_{\cC}(w_1) \neq N_{\cC}(w_2)$.
    % Without loss of generality, we may assume $e_1' = \{v_1,v_2,\ldots,v_{s-1}, w_1\}$ is an edge and $e_2' = \{v_2,\ldots,v_s, w_2\}$ is an edge.

    % Clearly, $e_1\setminus\{u_1\} \subseteq \{v_1,\ldots,v_s, w_1\}$ and $e_s\setminus\{u_s\} \subseteq \{v_1,\ldots,v_s, w_2\}$.
    % For $2 \le i \le s-1$, $e_i\setminus \{u_i\} \subseteq \{v_1,\ldots,v_s\}$.
    % Then $E(\cH)$ consists of
    % \begin{enumerate}
    %     \item $\cE$, $s$ hyperedges;
    %     \item $e$ containing some $u_i$, then $e\setminus \{u_i\} \subseteq \{v_1,\ldots,v_s\}$, at most $s \binom{s}{s-1}$ hyperedges;
    %     \item $e$ containing some $w \in V(\cH)\setminus \cC $, at most $n-2s$ hyperedges.
    % \end{enumerate}
    % In total,
    % \begin{equation*}
    % \begin{aligned}
    %     e(\cH) & \leq s + s\binom{s}{s-1} + (n-2s) \\
    %            & = 1 + s(n-s).
    % \end{aligned}
    % \end{equation*}

    % \textbf{Case 2.4:}
    % The only last case is that for every $v\in V(\cH)\setminus \cC$, $d_{\cC}(v)\leq 1$.
    % And we have
    % \begin{equation*}
    % \begin{aligned}
    %     e(\cH) & \leq s + \binom{2s}{s} + (n - 2s) \leq \max\{\binom{2s+1}{s}, 1 + s(n-s)\}.
    % \end{aligned}
    % \end{equation*}

    \subsection{\normalsize Proof of Theorem~\ref{thm: exact berge} for $s+1 \le r \le 2s+1$}
    % \vspace*{.4cm}

    \smallskip
    \noindent\textbf{Case 1.} $r=s+1$.
    \smallskip
    \smallskip

    By Claim~\ref{claim: 2 isolated vtx}, if $e(\cH)\leq \binom{2s+1}{r}$, then the result holds.
    So we assume that $w_1,w_2$ are two vertices  such that $d_{\cH'}(w_i)\geq 1$ for $i=1,2$.
    Then $|N_{\cH'}(w_i)|\geq r-1=s$.
    And $\overline{N_{\cH'}(w_1)}\cap N_{\cH'}(w_2)=\emptyset$, otherwise we can find a Berge copy of $M_{s+1}$. Thus, $|N_{\cH'}(w_1)| = |N_{\cH'}(w_2)| = s$.

    For the case $N_{\cH'}(w_1)\cap N_{\cH'}(w_2)=\emptyset$, we write $N_{\cH'}(w_1)=V_1$ and $N_{\cH'}(w_2)=V_2$. Then for every $x\in V_1$ and $y\in V_2$, there is no hyperedge $e \in E(\cH)\setminus \cE$ with $\{x,y\}\subseteq e$, otherwise we can find a $\cB M_{s+1}$. Therefore, $E(\cH'[\cC])=\empty$ and for every vertex $v\in V(\cH)\setminus \cC$, either $N_{\cH'}(v) \subseteq V_1$ or $N_{\cH'}(v) \subseteq V_2 $. Thus, $e(\cH)\leq s+n-2s = n-s$.

    For the case $N_{\cH'}(w_1)\cap N_{\cH'}(w_2) \neq \emptyset$,  there exists $u_j$~(or $v_j$) in $\overline{N_{\cH'}(w_1)}\cap \overline{N_{\cH'}(w_2)}$. Now we consider the hyperedge $e_j$ in $\cE$.
    Apart from $u_j$~(or $v_j$), $e_j$ does not contain vertices in $V(\cH)\setminus \cC$,  $\overline{N_{\cH'}(w_1)}$, and  $\overline{N_{\cH'}(w_2)}$ by a similar argument as Figure~\ref{fig: Berge exact}.
    Then we have
    \begin{align*}
    s = |e_j| -1 \le  |\cC \setminus (\overline{N_{\cH'}(w_1)} \cup \overline{N_{\cH'}(w_2)})| \le |\overline{N_{\cH'}(w_1)} \cap \overline{N_{\cH'}(w_2)}|.
    \end{align*}
     It follows that $\overline{N_{\cH'}(w_1)} = \overline{N_{\cH'}(w_2)}$ and thus $N_{\cH'}(w_1) = N_{\cH'}(w_2)$.
    This only happens when $N_{\cH'}(w_1)$ contains exactly one vertex from $\{u_i,v_i\}$.
    Without loss of generality, we may assume $N_{\cH'}(w_1) = N_{\cH'}(w_2) = \Omega=\{v_1,\dots,v_s\}$.
    In this case, for every $w \in V(\cH) \setminus \cC$, either $N_{\cH'}(w) = \Omega$ or $d_{\cH'}(w) = 0$.
    In addition, every hyperedge containing $u_i$~(including $e_i$) contains $\Omega$.
    Thus, we have
    \[
        e(\cH)\leq e(\cH[\Omega])+\sum_{u\in V(\cH)\setminus\Omega }d_{\cH'}(u)\leq \binom{s}{s+1}+(n-s)=n-s.
    \]

    \smallskip
    %\vspace*{.4cm}
    \noindent\textbf{Case 2.} $s+2\leq r\leq 2s$.
    \smallskip
    \smallskip

    Suppose that the result does not hold. By Claim~\ref{claim: 2 isolated vtx}, we may assume $w_1,w_2\in V(\cH)\setminus \cC$ with $d_{\cH'}(w_i)\geq 1$ for $i=1,2$. Then $N_{\cH'}(w_i)\geq r-1\geq s+1$. It implies $\overline{N_{\cH'}(w_1)}\cap N_{\cH'}(w_2)\neq \emptyset$. By a similar argument, we can find a $\cB M_{s+1}$, a contradiction.

    \smallskip
    %\vspace*{.4cm}
    \noindent\textbf{Case 3.} $r = 2s+1$.
    \smallskip
    \smallskip

    Suppose that the result does not hold, i.e., $e(\cH)\geq s+1$. Then there is at least one hyperedge $e \in E(\cH)\setminus \cE$. Since $r = 2s+1$, $e$ must be of the form $\{w\} \cup \cC$ for a $w \in V(\mathcal{H}) \setminus \cC$.
    Then $e_1$ contains a vertex outside $\{w\} \cup \cC$, and thus we may find a $\cB M_{s+1}$, a contradiction.

   % Suppose otherwise $e(\cH)\geq s+1$, then there are at least two hyperedges $e_1',e_2'\in E(\cH)\setminus \cE$. If $|e_i' \setminus \cC|\geq 2$, then we can find a Berge copy of $M_{s+1}$, for $i=1,2$ Thus, $\cC \subseteq e_i'$ for $i=1,2$ and let $x_i=e_i'\setminus \cC$ for $i=1,2$.
   % As a result, $(\cM\setminus e_1)\cup (e'_1\cup e_2')$ contains a Berge copy of $M_{s+1}$ with core $\cC \cup\{x_1,x_2\}$. A contradiction.
\end{proof}

\bigskip

\textbf{Funding}:

The research of Wang is supported by the China Scholarship Council (No.~202506210200) and the National Natural Science Foundation of China (Grant 12571372).

The research of Yang is supported by the National Natural Science Foundation of China (Nos.~12401464 and 12471334).

The research of Zhao is supported by the China Scholarship Council (No.~202506210250) and the National Natural Science Foundation of China (Grant 12571372).

The research of Bai is supported by the National Natural Science Foundation of China (Nos. 12131013 and 12471334), the Shaanxi Fundamental Science Research Project for Mathematics and Physics (No. 22JSZ009) and the China Scholarship Council (No. 202406290002).

The research of Zhou is supported by the China Scholarship Council (No.~202406890088) and the National Natural Science Foundation of China (Nos.~12271337 and 12371347).

\bigskip

\noindent{\bf{Declaration of interest}}

The authors declare no known conflicts of interest.

% ---------------- bib -------------------------------
\bibliography{ref.bib}
\bibliographystyle{wyc4}

\end{document}